\documentclass[12pt]{fsmfart}
\usepackage{color}
\usepackage{amssymb,verbatim}
\usepackage{amsthm,array,amssymb,amscd,amsfonts,amssymb,latexsym, url}
\usepackage{amsmath}
\usepackage[all]{xy}
\usepackage[french]{babel}
\usepackage{times}

 \usepackage{mathrsfs}  % allows nice script font for sheaves
\usepackage{bm}        % allows bold italic letters in math mode
\usepackage{mathtools} % allows more extendible arrows

 \newcounter{spec}
{\end{list}}

\renewcommand{\P}{{\mathbf P}}

\newcommand{\N}{{\mathbb N}}

\newcommand{\Z}{{\mathbb Z}}
\newcommand{\Q}{{\mathbb Q}}

\renewcommand{\L}{{\mathbb L}}

\newcommand{\oi}{\hskip1mm {\buildrel \simeq \over \rightarrow} \hskip1mm}

\newcommand{\Br}{{\operatorname{Br}}}

\newcommand{\Spec}{{\operatorname{Spec \ }}}

\renewcommand{\lim}{\varprojlim}

\renewcommand{\phi}{\varphi}

\numberwithin{equation}{section}

\newfont{\gothic}{eufb10}

\renewcommand{\qed}{{\hfill$\square$}}

\newtheorem{theo}{Th\'{e}or\`{e}me}[section]
\newtheorem{prop}[theo]{Proposition}

\newtheorem{lem}[theo]{Lemme}
\newtheorem{cor}[theo]{Corollaire}
\theoremstyle{definition}
\newtheorem{defi}[theo]{D\'efinition}
\theoremstyle{remark}
\newtheorem{rema}[theo]{Remarque}

\newcommand{\bthe}{\begin{theo}}
\newcommand{\ble}{\begin{lem}}
\newcommand{\bpr}{\begin{prop}}
\newcommand{\bco}{\begin{cor}}
\newcommand{\bde}{\begin{defi}}
\newcommand{\ethe}{\end{theo}}
\newcommand{\ele}{\end{lem}}
\newcommand{\epr}{\end{prop}}
\newcommand{\eco}{\end{cor}}
\newcommand{\ede}{\end{defi}}

\newcommand{\Gal}{{\rm{Gal}}}

\newcommand{\F}{{\mathbb F}}

\DeclareFontFamily{U}{wncy}{}
\DeclareFontShape{U}{wncy}{m}{n}{%
<5>wncyr5%
<6>wncyr6%
<7>wncyr7%
<8>wncyr8%
<9>wncyr9%
<10>wncyr10%
<11>wncyr10%
<12>wncyr6%
<14>wncyr7%
<17>wncyr8%
<20>wncyr10%
<25>wncyr10}{}
\DeclareMathAlphabet{\cyr}{U}{wncy}{m}{n}
\def\A{{\mathbb A}}

\def\G{{\mathbb G}}

\def\A{{\bf A}}
\def\Norm{{\rm Norm}}

\renewcommand\theenumi{\it\alph{enumi}}
\renewcommand\labelenumi{\theenumi)}

\begin{document}
%-----------------------------------------------------------------------%

\title[Non puissances $n$-i\`emes]{Le compl\'ementaire des puissances $n$-i\`emes dans un corps de nombres est un  ensemble diophantien}
\author{Jean-Louis Colliot-Th\'el\`ene}
 \address{CNRS \& Universit\'e Paris Sud\\Math\'ematiques, B\^atiment 425\\F-91405 Orsay Cedex\\France}
 \email{jlct@math.u-psud.fr}
 \author{Jan Van Geel}
\address{Universiteit Gent \\Vakgroep Wiskunde, Krijgslaan 281, S22
\\B-9000 Gent\\ Belgi\"e}
 \email{jvg@cage.ugent.be}
 
\date{5 janvier 2014}
\maketitle
\begin{abstract}
Le cas $n=2$ de l'\'enonc\'e du titre est un th\'eor\`eme de B. Poonen (2009), qui utilise une famille \`a un param\`etre
de vari\'et\'es et un th\'eor\`eme sur l'obstruction de Brauer-Manin pour les points rationnels de ces vari\'et\'es d\^u \`a Coray, Sansuc et l'un des auteurs (1980).
Pour $n=p$ premier quelconque, on dispose d'une famille  de vari\'et\'es analogues, qui ont \'et\'e consid\'er\'ees par
A. V\'arilly-Alvarado et B. Viray (2012). Nous obtenons le r\'esultat  annonc\'e en rempla\c cant ladite famille
par sa puissance sym\'etrique $(2p+1)$-i\`eme, et en appliquant un th\'eor\`eme sur l'obstruction de Brauer-Manin pour  les points rationnels de telles puissances sym\'etriques 
reposant sur des travaux de l'un des auteurs avec Swinnerton-Dyer (1994) et avec Skorobogatov et
Swinnerton-Dyer (1998),    qui  g\'en\'eralisent  un travail de Salberger (1988).

\end{abstract}
\begin{altabstract}
Voor $n=2$ is de bewering in de titel een
stelling van B. Poonen (2009). Hij gebruikt  een \'e\'en-parameter familie
van vari\"eteiten, en een stelling van Coray, Sansuc en \'e\'en van
de auteurs (1980), over de Brauer-Manin obstructie voor de rationale
punten van deze vari\"eteiten. Voor $n=p$, $p$ een willekeurig priemgetal, beschouwden A.
V\'arilly-Alvarado en B. Viray (2012)  een analoge familie van
vari\"eteiten. We bewijzen de bewering in de titel door
deze familie te vervangen door de $(2p+1)$-de symmetrische macht ervan en
door een stelling over de Brauer-Manin obstructie voor de rationale
punten van zulke symmetrische machten toe te passen.  Deze stelling steunt op werk van \'e\'en van de auteurs
met Swinnerton-Dyer (1994) en met Skorobogatov en Swinnerton-Dyer (1998).
Dat werk veralgemeent resultaten van Salberger (1988)
\end{altabstract}

\section{Introduction}

Soient $k$ un corps  et $n>0$ un entier. Un sous-ensemble $D$ de $\A^n(k)=k^{n}$ est dit diophantien s'il existe une 
$k$-vari\'et\'e $Z$ et un $k$-morphisme $f : Z \to \A^n_{k}$ tel que $D=f(X(k))$.
Dans cette d\'efinition, on peut supposer $Z$ affine. On peut  m\^eme supposer que
$Z$ est un sous-sch\'ema ferm\'e de $\A^{m+n}_{k}$ pour un certain $m$
et que le morphisme $f$ est induit par la projection $\A^{m+n}_{k} \to \A^n_{k}$.
Une union finie de sous-ensembles diophantiens dans $\A^n(k)$  est un sous-ensemble
diophantien. Nous  montrons (Th\'eor\`eme \ref{thprincipal}) :

{\it Pour tout corps de nombres $k$ et tout  entier $r>1 $,  
le compl\'ementaire dans $k^{\times}$ de l'ensemble $k^{\times r}$ des puissances $r$-i\`emes
est un ensemble diophantien.}

Pour $r$ une puissance de $2$, c'est un th\'eor\`eme de B. Poonen \cite{poonen}.
 Une   d\'emonstration ``\'el\'ementaire'' du  cas $r=2$ sur $k=\Q$  vient d'\^etre donn\'ee par
 J.~K\"onigsmann \cite[Prop. 20 (b)]{K}.

Dans le cas $r=2$, l'argument de Poonen utilise les surfaces de Ch\^atelet d'\'equation affine
$$y^2-az^2=P(x) $$
avec $a \in k^{\times}$ et  $P(x)$  polyn\^ome s\'eparable produit de deux polyn\^omes 
$Q(x)$ et $R(x)$ de degr\'e 2. Pour de telles surfaces, il est \'etabli par Colliot-Th\'el\`ene, Coray et Sansuc 
 \cite{CTCoSa}  que l'obstruction de Brauer-Manin est la seule obstruction au principe de Hasse pour les points rationnels sur de telles surfaces.  Ce r\'esultat est g\'en\'eralis\'e dans \cite{CTSaSD} \`a tout polyn\^ome $P(x)$ de degr\'e 4.
 
 Poonen \cite{poonen} part d'un contre-exemple au principe de Hasse d'\'equation
 $y^2-bz^2=Q(x)R(x)$ (il en existe \cite{poonen2}),  puis consid\`ere la famille ${\mathbb U}_{u}$ \`a un param\`etre $u \in k^{\times}$
 de surfaces $U_{u}$ d'\'equation
 $y^2-buz^2=Q(x)R(x)$.
 
 Un point cl\'e est que l'ensemble des $u \in k^{\times}$ sans obstruction de Brauer-Manin pour l'existence de points rationnels sur ${\mathbb U}_{u}$
 est un ensemble diophantien, car d'apr\`es \cite{CTCoSa}, c'est l'image des points $k$-rationnels
 de  ${\mathbb U}_{u}$ par le $k$-morphisme  ${\mathbb U}_{u} \to \G_{m,k}$ donn\'e par la coordonn\'ee $u$.
 
 Dans le cas $r=p$ avec $p$ premier quelconque, on consid\`ere les vari\'et\'es  d\'efinies
 sur un corps de nombres $k$ contenant une racine $p$-i\`eme de $1$ par une \'equation
  $${\rm Norm}_{k((bu)^{1/p})/k}(\Xi)=P(x)$$
  avec $b \in k^{\times}$ fixe, $u$ variant dans $ k^{\times}$, et
 $P(x)$ un polyn\^ome s\'eparable produit de deux polyn\^omes 
$Q(x)$ et $R(x)$ de degr\'e $p$. Sous l'hypoth\`ese de Bouniakowsky-Dickson-Schinzel,
 Colliot-Th\'el\`ene et Swinnerton-Dyer
  \cite{CTSD} montrent que  l'obstruction de Brauer-Manin est la seule obstruction au principe de Hasse pour les points rationnels de telles vari\'et\'es. C'est ce qui a permis  \`a A.~V\'arilly-Alvarado et B. Viray \cite{vav}
  d'\'etendre, de fa\c con conditionnelle, l'argument de Poonen \`a tout entier $r$.
  
  L'article  \cite{CTSD} \'etablit aussi un r\'esultat inconditionnel : pour des vari\'et\'es du type ci-dessus,
  l'obstruction de Brauer-Manin \`a l'existence d'un z\'ero-cycle de degr\'e 1 est la seule obstruction \cite[Thm. 5.1]{CTSD}. C'est une g\'en\'eralisation d'un r\'esultat de Salberger \cite{salbergerinv} sur les surfaces fibr\'ees en coniques.
  
  L'id\'ee nouvelle du pr\'esent article  est d'utiliser une version \og effective\fg \ de ce r\'esultat (Corollaire {\ref{caspartCTSD}) : Il existe un entier $N$, premier \`a $p$, ind\'ependant de $u$,  avec la propri\'et\'e suivante.
  Pour $u \in k^{\times}$, sur un mod\`ele projectif et lisse convenable ${\mathbb X}_{u}$ d'une vari\'et\'e  ci-dessus, s'il n'y a pas
  d'obstruction de Brauer-Manin \`a l'existence d'un point rationnel, alors
  il existe un z\'ero-cycle effectif de degr\'e $N$ sur ${\mathbb X}_{u}$, ce qui se traduit par le fait
  que le produit sym\'etrique ${\rm Sym}^N{\mathbb X}_{u}$ poss\`ede un $k$-point.
  Cela permet de montrer que le compl\'ementaire de l'ensemble des points $u \in k^{\times}$
pour lesquels  ${\mathbb X}_{u}(\A_{k}) \neq \emptyset$ et  ${\mathbb X}_{u}(\A_{k})^{\Br}=\emptyset$  est un ensemble diophantien. Le reste de l'argument  (Th\'eor\`eme \ref{thprincipal})
  est alors comme dans \cite{poonen}.
 
Pour \'etablir  le r\'esultat d'effectivit\'e,   nous  reprenons les d\'emonstrations de \cite{CTSD}, dans la version plus souple d\'evelopp\'ee dans \cite{CTSkSD} (Th\'eor\`eme \ref{caspartCTSkSD} et Corollaire
\ref{caspartCTSD}). La m\'ethode donne  $N=2p+1$. Comme nous l'a signal\'e O.~Wittenberg au vu d'une pr\'ec\'edente version de cet article, un r\'esultat
d'effectivit\'e plus g\'en\'eral a d\'ej\`a \'et\'e obtenu, par la m\^eme m\'ethode, dans son article \cite{w}.

Si l'on admet un r\'esultat alg\'ebrique annonc\'e par Salberger dans sa th\`ese \cite{salbergerthese}, mais non publi\'e, on peut se dispenser de revisiter les d\'emonstrations des  articles \cite{CTSD} ou \cite{CTSkSD}, ou de citer \cite{w}. Cette m\'ethode alternative donne  $N=(p-1)^2$. Le r\'esultat de Salberger  \cite{salbergerthese} est d\'ecrit au \S 6, la variante de la d\'emonstration du
 th\'eor\`eme principal \'etant expos\'ee au \S 5.

\bigskip
\bigskip

{\bf Rappels et notations}

On note $A[n]$ le sous-groupe de $n$-torsion d'un groupe ab\'elien $A$.

 Soient $k$ un corps et $n$ un entier. On appelle extension  cyclique $K$ de $k$  de degr\'e $n$
  une $k$-alg\`ebre \'etale $K$ et une action de $G=\Z/n$ sur $K$ qui fait de $\Spec K \to \Spec k$ 
  un $G$-torseur. En particulier par extension cyclique de corps $K/k$, on entend dans ce texte une extension cyclique
  galoisienne de groupe $G=\Z/n$ muni du g\'en\'erateur $1 \in \Z/n$. Ceci d\'efinit une classe $\chi_{K/k} \in
  H^1(k,\Z/n)$. Si $n \neq 0 \in k$, \`a tout \'el\'ement $c \in k^{\times}$
  on associe sa classe dans $k^{\times}/k^{\times n} \oi H^1(k,\mu_{n})$.

  Pour $K/k$ une extension cyclique de degr\'e $n$ premier \`a la caract\'eristique de $k$,
  et $c \in k^{\times}$, 
  on dispose de l'alg\`ebre simple centrale  cyclique $(K/k,c)$ de degr\'e $n$, dont on note encore
  $(K/k,c)$ la classe dans le groupe de Brauer
$\Br k$,  qui est d\'efinie comme
  le cup-produit via  $$ H^1(k,\Z/n) \times H^1(k,\mu_{n}) \to H^2(k,\mu_{n})$$
  de la classe $\chi_{K/k}$ et de la classe de $c$ dans $k^{\times}/k^{\times n}$.
  
 Supposons que 
  $k$ contient une racine primitive $n$-i\`eme de $1$, soit $\zeta$.
  Le choix d'un isomorphisme $\Z/n \oi \mu_{n}$ permet d'identifier
  $H^{1}(k,\Z/n)=H^1(k,\mu_{n})=k^{\times}/k^{\times n}$ et $(\Br k )][n]=H^2(k,\mu_{n}) = H^2(k,\mu_{n}^{\otimes 2})$.
  \'Etant donn\'es $b,c \in k^{\times}$,  on note alors $(b,c)_{\zeta} \in (\Br k) [n]$ le cup-produit  des classes
   $b$ et $c$ dans $k^{\times}/k^{\times n} =H^1(k,\mu_{n})$. Compte tenu de l'identification 
   $\Z/n=\mu_{n}$, la
  $k$-alg\`ebre $k(b^{1/n}):=k[t]/(t^n-b)$
   est munie naturellement d'une structure d'extension cyclique de degr\'e $n$, et l'on a $(k(b^{1/n})/k,c)=(b,c)_{\zeta}  \in \Br k$.
 
La lecture du pr\'esent   article requiert  une certaine  familiarit\'e avec les articles \cite{CTSD}, \cite{CTSkSD}, \cite{poonen}.

\medskip

 \section{Alg\`ebre}\label{enfamille}

 Le lemme suivant est bien connu (cf. \cite[Chapitres 4 et  5]{GS}).
   \begin{lem}\label{connu}
  Soit $K/k$ une extension cyclique de corps de degr\'e $n$ premier \`a la caract\'eristique de  $k$.  Soit $c \in k^{\times}$. 
  
  (i) La $k$-vari\'et\'e affine  $Y$ 
  d'\'equation ${\rm Norm}_{K/k}(\Xi)=c$
  est un ouvert de la  $k$-vari\'et\'e de Severi-Brauer $X$  d'indice $n-1$
  attach\'ee \`a   l'alg\`ebre simple centrale $(K/k,c)$.

  (ii)  La $k$-vari\'et\'e  $Y$ poss\`ede un $k$-point si et seulement si la classe de $(K/k,c)$ est nulle dans $\Br k$.
  
  (iii) On a une suite exacte $$   \Z/n \to \Br k \to \Br X \to 0,$$
  o\`u $1 \in \Z/n$ a pour image $(K/k,c) \in \Br k$.
   \end{lem}
  
  \begin{proof} En utilisant \cite{GS}, on \'etablit ce lemme bien connu, \`a un point pr\`es.
 L'exercice \cite[Ex. 1]{GS}, o\`u il convient de remplacer $bx$ par $bx^n$,
donne un  \'enonc\'e d'\'equivalence birationnelle stable  au lieu de (i) ci-dessus.
Esquissons comment l'on obtient (i).  
On a les suites exactes de $k$-tores
$$ 1 \to R^1_{K/k}\G_{m} \to R_{K/k}\G_{m} \stackrel{\Norm_{K/k}}{\longrightarrow}  \G_{m,k} \to 1$$
et
$$ 1 \to \G_{m,k}\stackrel{x\mapsto x}{{\longrightarrow}}  R_{K/k}\G_{m} \to R_{K/k}\G_{m}/\G_{m,k } \to 1.$$
Notons $T_{1}=R^1_{K/k}\G_{m}$ et $T= R_{K/k}\G_{m}/\G_{m,k }$.
Le choix d'un g\'en\'erateur de $\Gal(K/k)$ d\'efinit un isomorphisme de $k$-tores $T_{1} \oi T$.
 La $k$-vari\'et\'e $Y$ est un espace principal homog\`ene sous le $k$-tore $T_{1}$.
 C'est donc aussi un espace principal homog\`ene sous le $k$-tore $T$.
  On a un plongement  torique, de $R_{K/k}\G_{m}$ dans $W:=R_{K/k}\G_{a} \setminus \{0\} \simeq \A^n_{k}  \setminus \{0\}$,
 qui induit un  plongement torique de $T$ dans le quotient de $W$ par l'action diagonale de $\G_{m,k}$,
 quotient qui s'identifie \`a l'espace projectif $\P^{n-1}_{k}$. La $k$-vari\'et\'e $Y$
 est donc un ouvert dans la $k$-vari\'et\'e  $X= Y\times^{T}\P^{n-1}_{k}$ quotient de $Y \times \P^{n-1}_{k}$
 par l'action diagonale de $T$. Cette $k$-vari\'et\'e $X$ est  une forme tordue de l'espace projectif.
 On a le diagramme commutatif de $k$-groupes
 $$
 \begin{CD}
  1 & \to & \G_{m,k}& \to & R_{K/k}\G_{m} & \to & T& \to & 1\cr
  && @VV{\rm id}V @VVV @VVV   \cr
  1 & \to & \G_{m,k} & \to & GL_{n,k} & \to &PGL_{n,k}& \to & 1.
  \end{CD}
  $$ 
  Ce diagramme induit un diagramme commutatif
 $$
 \begin{CD}  H^1(k, T)& \to & \Br k\cr
 @VVV @VV{\rm id}V \cr
   H^1(k, PGL_{n,k} )& \to &  \Br k.
    \end{CD}
  $$
  La fl\`eche $H^1(k,T) \to H^1(k,PGL_{n,k})$ envoie la classe de $Y$ sur la classe de $X$, et 
 le compos\'e $k^{\times}/N_{K/k}K^{\times} \oi  H^1(k,T_{1}) \oi H^1(k,T) \to \Br k$ envoie la classe 
  de $c$ sur la classe de l'alg\`ebre cyclique $(K/k,c)$ (cf. \cite[Cor. 4.7.4]{GS}). Quant \`a la fl\`eche
   $ H^1(k, PGL_{n,k} )\to  \Br k$, elle envoie la classe d'isomorphie d'une vari\'et\'e de Severi-Brauer 
   d'indice $n-1$ sur sa classe dans le groupe de Brauer.
   \end{proof}

\begin{prop}\label{modelecalculbrauer}
Soit $k$ un corps de caract\'eristique z\'ero.
 Soient  $p$ un nombre premier
 et  $P(x) \in k[x]$ un polyn\^ome s\'eparable de degr\'e $2p$. 
 Soit $K/k$ une extension galoisienne cyclique, de groupe $G=   \Z/p$.
 Soit $U$ la $k$-vari\'et\'e d\'efinie par
 $${\rm Norm}_{K/k}(\Xi)=P(x) \neq 0.$$

(i)  Il existe une $k$-compactification lisse $X$ de $U$, 
 \'equip\'ee d'un morphisme
  $\pi : X \to \P^1_{k}$ \'etendant l'application $(\Xi,x) \mapsto x$,
  dont la fibre g\'en\'erique $X_{\eta}/k(\P^1)$ est une vari\'et\'e de Severi-Brauer
  de dimension $p-1$, d'alg\`ebre simple centrale associ\'ee l'alg\`ebre cyclique $(K/k,P(x))$, telle
  que pour $M$ point ferm\'e de $\P^1_{k}$ non z\'ero de $P(x)$,
  la fibre $X_{M}$ est une vari\'et\'e de Severi-Brauer sur le corps r\'esiduel $k(M)$.

(ii) Si $K$ n'est pas un corps, alors  la $k$-vari\'et\'e $X$ est $k$-birationnelle \`a $\P^p_{k}$, et $\Br k = \Br X$.

(iii)  Supposons que $K$ est un corps, et que $P(x)=Q(x).R(x)$ avec $Q(x)$ et $R(x)$
irr\'eductibles de degr\'e $p$ et sans z\'ero dans $K$. Alors les alg\`ebres cycliques
$(K/k,Q(x))$ et $(K/k,R(x))$ dans $\Br k(x)$ ont une image dans $\Br k(X)$ qui appartient \`a $\Br X$,
et  le quotient de $\Br X$ par l'image de $\Br k$ est un groupe cyclique d'ordre $p$ engendr\'e par  l'image de $(K/k,Q(x))$.
\end{prop}

\begin{proof}\footnote{La 
 proposition \cite[Prop. 2.1]{vav} est incorrecte;  nous
  ne pouvons donc   citer \cite[\S 3]{vav}.}
On consid\`ere la $k$-vari\'et\'e d'\'equation
$${\rm Norm}_{K/k}(\Xi)=P(x)$$
et $U_{1}\subset U$ l'ouvert maximal sur lequel la projection $U_{1 }\to \A^1_{k}=\Spec k[x]$
d\'efinie par $x$ est lisse.  

 Soit $V$ la $k$-vari\'et\'e $V$
d'\'equation
$${\rm Norm}_{K/k}(\Theta) = P'(y)$$ 
avec $P'(y):=y^{2p}P(1/y)$.
Soit $V_{1} \subset V$ l'ouvert maximal sur lequel le morphisme $V_{1}    \to \A^1_{k}=\Spec k[y]$
d\'efini par $y$ est lisse.

On   recolle la $k$-vari\'et\'e $U_{1}$ et la $k$-vari\'et\'e $V_{1}$ en une $k$-vari\'et\'e $W$
  via $x=1/y$ et $y^2.\Xi= \Theta$.
On dispose donc d'un $k$-morphisme $W \to \P^1_{k}$ \`a fibres lisses, dont toutes
les fibres sauf celles au-dessus des points ferm\'es \`a support dans $P(x)=0$
sont g\'eom\'etriquement int\`egres.

Une variante du lemme \ref{connu} permet alors de construire 
une $k$-vari\'et\'e $W'$ lisse \'equip\'ee d'un morphisme $W' \to \P^1_{k}$
dont les fibres au-dessus des points ferm\'es autres que ceux support\'es dans $P(x)=0$
sont des vari\'et\'es de Severi-Brauer, et dont les fibres en les points ferm\'es support\'es
dans $P(x)=0$ 
sont les fibres de $W \to \P^1_{k}$.
En utilisant le th\'eor\`eme d'Hironaka, on obtient alors l'existence d'une $k$-vari\'et\'e projective 
$X$ munie d'un k-morphisme $\pi : X \to \P^1_{k}$, contenant $W'$ comme ouvert, le morphisme $\pi$
\'etendant $W' \to \P^1_{k}$.

L'\'enonc\'e (ii) est \'evident : si $K$ n'est pas un corps, alors ${\rm Norm}_{K/k}(\Xi)$
s'\'ecrit comme un produit de variables ind\'ependantes, et la $k$-vari\'et\'e d\'efinie par  l'\'equation
$$Y_{1} \dots Y_{p} = P(x) \neq 0$$
est clairement $k$-rationnelle.

L'\'enonc\'e (iii)  est  essentiellement d\'emontr\'e dans  \cite[\S 2.2, Thm. 2.2.1]{CTSD}.
 Nous donnons les points principaux de la d\'emonstration, renvoyant le lecteur
 \`a \cite[\S 1 et \S 2]{CTSD} pour des explications plus d\'etaill\'ees.
 
 On fixe  un g\'en\'erateur $\sigma$ de $\Gal(K/k)$, et on note  $(K/k,c)=(K/k,\sigma,c)$.
 
 On a la suite exacte  (Lemme \ref{connu}) 
 $$\Z/p.(\alpha) \to \Br k(\P^1)\stackrel{\pi^*}{\longrightarrow}  \Br X_{\eta} \to 0,$$
 o\`u $\alpha=(K(x)/k(x),P(x))$, qu'on note $(K/k,P(x))$.
Soit   $\gamma \in \Br X \subset  \Br X_{\eta}$. Il existe donc  $\beta \in  \Br k(\P^1)$
tel que $\gamma=\pi^*(\beta)$.
 
Les fibres de $\pi$ aux points ferm\'es de $\A^1_{k}=\Spec k[x]$ autres que ceux d\'efinis par $Q=0$
et $R=0$ sont lisses et g\'eom\'etriquement int\`egres. La comparaison des r\'esidus de $\gamma$ et de $\beta$  entra\^{\i}ne que
les r\'esidus de $\beta$ en tout point ferm\'e de  $\A^1_{k}$ autre que $Q=0$ et $R=0$
sont nuls.  La comparaison des r\'esidus de $\gamma$ et de $\beta$ au-dessus de  $Q=0$, resp. de $R=0$,  
sur l'ouvert $U_{1}\subset X$,
montre aussi que ${\rm Res}_{Q} (\beta) \in H^1(k_{Q},\Q/\Z)$,
resp. ${\rm Res}_{R} (\beta) \in H^1(k_{R},\Q/\Z)$, o\`u $k_{Q}$, resp. $k_{R}$, d\'esigne le corps r\'esiduel
en $Q=0$, resp. en $R=0$, s'annulent par passage \`a $K\otimes_{k}k_{Q}$, resp. $K\otimes_{k}k_{R}$.
Les groupes  $H^1(K\otimes_{k}k_{Q}/k_{Q},\Q/\Z) \simeq \Z/p$, et $H^1(K\otimes_{k}k_{R}/k_{R},\Q/\Z)\simeq \Z/p$
sont engendr\'es par l'image d'un g\'en\'erateur $\chi$  de $H^1(K/k,\Q/\Z)\simeq \Z/p$.
La $k$-alg\`ebre $(K/k,Q(x))$ a  sur $\A^1_{k}$ pour seul r\'esidu $\chi_{k_{Q}}$ en $Q=0$.
La $k$-alg\`ebre $(K/k,R(x))$ a  sur $\A^1_{k}$ pour seul r\'esidu $\chi_{k_{R}}$ en $R=0$.
On voit donc qu'il existe des entiers $r$ et $s$ tels que $\beta - (K/k,Q(x)^r) -(K/k,R(x)^s)$ ait tous 
ses r\'esidus triviaux sur $\A^1_{k}$, donc par la suite exacte de Faddeev  (cf. \cite[\S 1.2]{CTSD})
soit dans l'image de $\Br k$.
Comme la classe $\alpha= (K/k,P(x)) \in \Br k(\P^1)$ a une image nulle dans $\Br k(X)$, on voit
que $\gamma \in \Br X \subset \Br k(X)$ est la somme d'un multiple de $\pi^*(K/k,Q(x))$ et d'une classe
dans $\Br k$. Nous renvoyons \`a \cite[Thm. 2.2.1]{CTSD}  pour la d\'emonstration du fait que
$\pi^*(K/k,Q(x))$ appartient \`a $\Br X$, d\'emonstration qui utilise le fait que le degr\'e de $Q$ est $p$. 
Si  la classe $\pi^* (K/k,Q(x)) \in \Br k(X)$ appartenait \`a l'image de $\Br k$, d'apr\`es le lemme~\ref{connu}
il existerait un entier $r$ tel que $$(K/k,Q(x)) -r (K/k,P(x)) = (K/k,Q(x))
- r(K/k,Q(x)R(x)) \in \Br k(x)$$
 appartienne \`a $\Br k$.
 Le calcul du r\'esidu d'un tel \'el\'ement au point ferm\'e $R(x)=0$ donne $r=0$. Et le r\'esidu 
 de $(K/k,Q(x))$ en $Q(x)=0$ n'est pas nul. Ainsi $\Br X/\Br k$ est d'ordre $p$, engendr\'e
 par la classe $\pi^*(K/k,Q(x))$.
  Le calcul de r\'esidu a utilis\'e le fait
que ni $Q$ ni $R$ n'ont de z\'ero dans $K$.
\end{proof}

\section{Arithm\'etique}\label{arith}

\medskip

Soient $k$ un corps de nombres et  $\Omega$ l'ensemble de ses places.
Pour $v\in \Omega$, on note $k_{v}$ le compl\'et\'e de $k$ en $v$
et pour $v$ non archim\'edienne, $\F_{v}$ le corps r\'esiduel en $v$.

Comme mentionn\'e dans l'introduction, le th\'eor\`eme suivant, avec pr\'ecis\'ement la m\^eme borne $d\geq d_{0}$, est un cas particulier d'un  r\'esultat de Wittenberg \cite[Thm. 4.8]{w},
r\'esultat qui  est une cons\'equence de \cite[Thm. 4.1]{CTSkSD}.  Comme nous partons ici d'une hypoth\`ese
sur les points rationnels et non sur les z\'ero-cycles, la d\'emonstration que nous proposons est un peu plus simple que celle de \cite[Thm. 4.8]{w}.
Comme elle est courte, nous l'avons conserv\'ee.

\begin{theo}\label{caspartCTSkSD}
Soit  $X$ une $k$-vari\'et\'e projective, lisse, g\'eom\'etriquement int\`egre sur un corps de nombres $k$,
\'equip\'ee d'un morphisme projectif et  plat $\pi : X \to \P^1_{k}$ \`a fibre g\'en\'erique g\'eom\'etriquement int\`egre.
Soit $d_{0}$ la somme des degr\'es sur $k$ des points ferm\'es de $\P^1_{k}$ dont la fibre n'est pas lisse.

Supposons :

(i) Pour tout point ferm\'e $M \in \P^1_{k}$, de corps r\'esiduel $k(M)$, la fibre $X_{M}/k(M)$ contient une composante irr\'eductible $Z \subset X_{M}$, de multiplicit\'e 1,
  telle que la fermeture alg\'ebrique de $k(M)$ dans le corps des fonctions de $Z$
est une extension ab\'elienne de $k(M)$.

(ii) Il existe un ad\`ele dans $ X(\A_{k})$ orthogonal au groupe de Brauer vertical de $X$,
c'est-\`a-dire au sous-groupe de $\Br X$ dont l'image dans $\Br k(X)$ appartient \`a $\pi^*\Br k(\P^1)$.

(iii) Le principe de Hasse vaut pour les fibres lisses de $\pi$  au-dessus de tout point ferm\'e de $\P^1_{k}$.

Alors, pour tout entier $d \geq d_{0}$, et tout ouvert de Zariski non vide $V \subset \P^1_{k}$, il existe
un point ferm\'e $m \in V$ de degr\'e $d$ sur $k$,  \`a fibre  $X_{m}$ lisse contenant un $k(m)$-point rationnel.
\end{theo}
\begin{proof} 
On peut supposer que la fibre de $\pi$ au-dessus du point  \`a l'infini de $\P^1_{k}$ est lisse.
Soit $V \subset \A^1_{k}$ le compl\'ementaire de l'ensemble des points ferm\'es $M$  dont la fibre $X_{M}/k(M)$  n'est pas lisse.
Soit $U=\pi^{-1}(V)$. La projection $\pi: U  \to V$ est donc lisse. 
Il suffit de modifier la d\'emonstration de  \cite[Thm. 4.1]{CTSkSD}. Nous supposons que le lecteur
a le texte \cite{CTSkSD} sous la main.

\`A la page 19 de \cite{CTSkSD},   au lieu d'appliquer \cite[Thm. 3.2.2]{CTSD} aux z\'ero-cycles,
on applique  \cite[Thm. 3.2.1]{CTSD} aux points rationnels, c'est-\`a-dire le lemme formel d'Harari, comme c'est fait
dans \cite[(1.4)]{CTSkSD}. 
C'est-\`a-dire que l'on commence comme dans la d\'emonstration de   \cite[Thm. 1.1]{CTSkSD}, plus pr\'ecis\'ement comme au premier paragraphe de \cite[p. 7]{CTSkSD}.

Ceci donne des \'egalit\'es
$$ \sum_{v \in S_{1}}A_{i,j} (p_{v}) = 0 \in \Q/\Z,$$
avec des $p_{v} \in U(k_{v})$ et $S_{1}$ un ensemble fini de places contenant
l'ensemble $S$ des places ``de mauvaise r\'eduction''.

En utilisant le th\'eor\`eme des fonctions implicites, pour chaque place $v$  on peut trouver
des $p_{v,l} \in U(k_{v})$, $l=1, \dots,d$,  d'images respectives $m_{v,l}=\pi(p_{v,l})$, tous distincts
dans $V(k_{v})$, 
tels que $A_{i,j} (p_{v,l}) =A_{i,j} (p_{v}) \in \Br k_{v}$ pour tout $l$.  
 
 On reprend alors la d\'emonstration de \cite[Thm. 4.1]{CTSkSD}, page 19,  \`a la formule (4.3),
 en d\'efinissant pour  chaque $v\in S_{1}$ le  z\'ero-cycle $z_{v}^2$ sur $U$ 
  $$z_{v}^2 = \sum_{i=1}^d p_{v,l}.$$
 
On continue alors la d\'emonstration, avec des polyn\^omes $G_{v}(t)$ unitaires
de diviseur $\pi(z_{v}^2)$ sur $\A^1_{k_{v}}$. Ces polyn\^omes sont 
 de degr\'e $d$ (au lieu de $1+Dds$ avec les notations de \cite{CTSkSD}).
 Proc\'edant comme au troisi\`eme paragraphe de la page 20 de \cite{CTSkSD},  
en utilisant l'astuce de Salberger, comme d\'etaill\'ee au \S3 de \cite{CTSkSD}),
on trouve un polyn\^ome irr\'eductible $G(t) \in k[t]$, de degr\'e $d$, d\'efinissant
un point ferm\'e $m \in V$ dont la fibre $X_{m}/k(m)$ a des points dans tous les compl\'et\'es de $k(m)$, 
donc a un $k(m)$-point rationnel.
\end{proof}

Soit $k$ un corps parfait, $\overline{k}$ une cl\^oture s\'eparable de $k$.   
Pour $N> 0$ entier, et toute $k$-vari\'et\'e quasi-projective $U$,
on note ${\rm Sym}^NU$  le quotient de l'action du groupe sym\'etrique $\frak{S}_{N}$
sur le produit   $U^N$. On a une bijection
 entre les ensembles suivants :

(i) Les $k$-points de ${\rm Sym}^NU$.

(ii) Les   z\'ero-cycles effectifs de degr\'e $N$ sur $U$.

(iii) Les z\'ero-cycles effectifs de degr\'e $N$ sur $U\times_{k}{\overline k}$
qui sont invariants sous l'action de $ \Gal({\overline k}/k)$.

La bijection entre (ii) et (iii) est claire.
Pour  tout $k$-point de  ${\rm Sym}^NU$ on choisit un point $(P_{1}, \dots, P_{N}) \in U^N(\overline{k})$.
Le z\'ero-cycle $P_{1}+\dots+P_{N}$ est invariant sous $ \Gal({\overline k}/k)$.
Inversement si  $P_{1}+\dots+P_{N}$ est un z\'ero-cycle  effectif de degr\'e $N$ sur $U\times_{k}{\overline k}$
 invariant sous l'action de $ \Gal({\overline k}/k)$, alors  l'image de $(P_{1}, \dots, P_{N}) \in U^N(\overline{k})$
 dans  $({\rm Sym}^NU)({\overline k})$ est invariante sous $ \Gal({\overline k}/k)$, donc 
 d\'efinit un $k$-point de ${\rm Sym}^NU$.

\begin{cor}\label{caspartCTSD}
  Soit $k$ un corps  de nombres. Soit $K/k$ une extension finie cyclique de corps, de degr\'e $p$.
  Soit $P(x)=Q(x)R(x)$ un polyn\^ome s\'eparable de degr\'e $2p$ produit de deux polyn\^omes irr\'eductibles
  de degr\'e $p$. Soit $U=U(K/k,P)$ la $k$-vari\'et\'e affine,  lisse, int\`egre d\'efinie par
  l'\'equation
  $$ {\rm Norm}_{K/k}(\Xi)= P(x) \neq 0.$$
Soit $X$ une $k$-vari\'et\'e projective et lisse, g\'eom\'etriquement int\`egre, contenant $U$
  comme ouvert dense. 
  S'il existe un \'el\'ement 
   $\{M_{v}\}$ 
  du produit $\prod_{v \in \Omega} U(k_{v})$ orthogonal \`a $(K/k,Q(x))  \in \Br X$,
  alors il existe une extension de corps $L/k$ de degr\'e $2p+1$
  avec $U(L) \neq \emptyset$. En particulier  les $k$-vari\'et\'es ${\rm Sym}^{2p+1}U$ et ${\rm Sym}^{2p+1}X$
     poss\`edent un $k$-point.
 \end{cor}

\begin{proof} 
D'apr\`es la proposition \ref{modelecalculbrauer},   il existe un tel mod\`ele $X$.
Les propri\'et\'es d'invariance birationnelle du groupe de Brauer montrent que l'\'enonc\'e
ne d\'epend pas du choix du mod\`ele. D'apr\`es la proposition \ref{modelecalculbrauer},   les classes 
   $(K/k,Q(x))_{k(X)}$ et  $(K/k,R(x))_{k(X)}$ sont chacune dans $\Br X$,
leur somme dans $\Br k(X)$ est dans l'image de $\Br k$, et ce sont des g\'en\'erateurs
de $\Br X$. Les hypoth\`eses du th\'eor\`eme \ref{caspartCTSkSD} sont clairement satisfaites,
avec $d_{0}=2p$. On choisit ici $d=2p+1$.
\end{proof}

\begin{rema} Dans le  cadre du corollaire, les seuls $A_{i,j}$ intervenant dans la preuve du
th\'eor\`eme \ref{caspartCTSkSD} sont  $(K/k,Q(x))_{k(X)}$ et $(K/k,R(x))_{k(X)}$,
et ces \'el\'ements de $\Br k(X)$ sont ici dans $\Br X$.
On n'a   donc pas  besoin ici d'invoquer le lemme formel d'Harari, l'hypoth\`ese
donne les \'egalit\'es
$$ \sum_{v \in S_{1}}A_{i,j} (M_{v}) = 0 \in \Q/\Z,$$
avec $M_{v}$ comme dans l'\'enonc\'e du corollaire, et
avec $S_{1}$ \'egal \`a l'ensemble fini $S$  \'evident des places de mauvaise r\'eduction.

Si l'on se limite \`a $p$ premier impair, ce qui suffirait ici, on peut aussi
\'eviter les difficult\'es sp\'ecifiques aux   places r\'eelles (voir \`a ce sujet \cite[p. 82, l. 6/8]{CTSD} et \cite[p. 17, l. 9/11]{CTSkSD}).
\end{rema}

\begin{prop}\label{exemple}
Soient  $p$ un nombre premier   et $k$ un corps de nombres  contenant une racine
primitive $p$-i\`eme de l'unit\'e $\zeta$. 
 Il existe $d \in k^{\times}$, $d \notin k^{\times p} $ et
 un polyn\^ome s\'eparable $P(x)=Q(x)R(x) \in k[x]$
avec $Q(x)$ et $R(x)$ irr\'eductibles de degr\'e $p$, sans z\'ero dans $K=k(d^{1/p})$, 
tels que toute $k$-vari\'et\'e $X$ projective, lisse, g\'eom\'etriquement connexe
$k$-birationnelle \`a la $k$-vari\'et\'e $U=U(K/k,P)$ d'\'equation affine
$$ {\rm Norm}_{K/k}(\Xi)= P(x) \neq 0$$
satisfasse :

(i) $X$ poss\`ede des points rationnels dans tous les compl\'et\'es $k_{v}$ de $k$.

(ii) $X$ ne poss\`ede pas  de z\'ero-cycle de degr\'e 1. Plus pr\'ecis\'ement, l'alg\`ebre
 $A=(K/k,Q(x))=(d,Q(x))_{\zeta}  \in \Br k(X)$ appartient \`a $\Br X$ et, pour toute famille 
 $\{z_{v}\}, v \in \Omega$, de z\'ero-cycles de degr\'e 1 sur $X$, on a
 $$\sum_{v \in \Omega} A(z_{v}) \neq 0 \in \Q/\Z.$$
 
 (iii)   On a $({\rm Sym}^{2p+1}U)(k)=\emptyset$ et  $({\rm Sym}^{2p+1}X)(k)=\emptyset$.
 
 \end{prop}

\begin{proof} Si une $k$-vari\'et\'e $X$ satisfait (i) et (ii), toute autre
$k$-vari\'et\'e projective, lisse, g\'eom\'etriquement connexe $k$-birationnelle \`a $X$
satisfait (i) et (ii).

L'\'enonc\'e (iii) est une cons\'equence imm\'ediate de l'\'enonc\'e (ii).
En effet $U$ poss\`ede de fa\c con \'evidente un point dans une extension de degr\'e $p$.

Poonen \cite[\S 5, Prop. 5.1]{poonen2} (cas $p=2$) 
puis de fa\c con semblable  V\'arilly-Alvarado et Viray \cite[Prop. 4.1]{vav} (cas $p$ impair)
construisent (au moins biratonnellement) une telle vari\'et\'e $X$ avec des points rationnels dans tous les $k_{v}$ mais telle
qu'il y ait obstruction de Brauer-Manin \`a l'existence d'un point rationnel. Nous allons montrer que
  leur exemple satisfait aussi l'\'enonc\'e analogue pour les z\'ero-cycles de degr\'e~1, \`a savoir 
  l'\'enonc\'e (ii) de la proposition.
  Ceci ne r\'esulte pas formellement du cas des points rationnels   (cf. \cite[\S 10]{CTSD}).
  
  Soit  $M\in \N$ un entier tel que toute courbe plane projective lisse (et donc g\'eom\'etriquement int\`egre)
 de degr\'e $p$  sur un corps fini $\F$  de cardinal au moins $M$
poss\`ede au moins $2p+1$ points $\F$-rationnels. 
 
En utilisant le th\'eor\`eme de densit\'e de Chebotarev et la th\'eorie du corps de classes, 
on trouve  $a, b,c \in k$ avec les propri\'et\'es suivantes: 
\begin{itemize}
\item[(1)] $b\in O_k$, $bO_k$ id\'eal  premier, $b\equiv 1\bmod (1-\zeta)^{2p-1}O_k$, $b>>0$  et $\#\F_b>M$,
\item[(2)] $a\in O_k$, $aO_k$   id\'eal premier distinct de $bO_k$, $a \equiv 1\bmod (1-\zeta)^{2p-1}O_k$, $a \not\in k_{b}^{\times p}$, $a>>0$ et $\#\F_a>M$.
\item[(3)] $c \in O_k$, $b\; | \; ac+1$.
\end{itemize}
Soient  $d=ab$ et $K$ le corps $k(d^{1/p})$.   Soient  $Q(x)=x^p+c$ et $R(x)=ax^p+ac+1$.

Soit $X$ une $k$-vari\'et\'e projective, lisse, g\'eom\'etriquement connexe et $k$-birationnelle \`a la $k$-vari\'et\'e
 $U$
d'\'equation 
$$ \Norm_{K/k}(\Xi)= Q(x) R(x)\neq 0.$$

\medskip

Suivant  \cite[Lemma 5.3]{poonen2},
montrons que l'on a 
  $U(k_{v})\neq \emptyset$  pour toute place~$v$.

Comme on a $d>>0$, on a $U(k_{v})\neq \emptyset$ pour toute place archim\'edienne.

Soit $v$ une place finie de $k$ diff\'erente de  $v=v_{a}$, $v=v_{b}$ et telle que $v(p)=0.$
 L'extension $k(d^{1/p})/k$ est non ramifi\'ee en $v$.  Pour tout $x \in k_{v}$ avec $v(x)<0$,
 on a $v(P(x))= 2p$, on a donc des points dans $U(k_{v})$ avec un tel $x$. 
 
 Soit $v$ une place avec $v(p)>0$. L'entier $d$ satisfait $d=ab \equiv 1\bmod (1-\zeta)^{2p-1}O_k$.
 Soit $h(x)=x^p-d$. Les \'el\'ements $p$ et $(1-\zeta)^{p-1}$ diff\`erent par une unit\'e dans $\Q(\zeta)$.
 On en d\'eduit $v(h(1))>2v(h'(1))$, o\`u $h'(x)=px^{p-1}$ est le polyn\^ome d\'eriv\'e de $h(x)$.
 Le lemme de Hensel assure 
 alors l'existence d'une solution (enti\`ere et congrue \`a $1$ modulo $v$) de $x^p-d=0$ dans $k_{v} $.
 On a donc   $U(k_{v})\neq \emptyset$.

 Soit $v=v_{b}$. Alors  $a$ et $c$ sont des  unit\'es dans $k_{v}$. 
Les hypoth\`eses
assurent que
 la courbe  affine lisse $z^p=a(x^p+c)\neq 0$ sur le corps $\F_{v}$ poss\`ede un
 point $\F_{v}$-rationnel. Il en est donc de m\^eme de la courbe $z^p=(x^p+c) (ax^p+ac+1) \neq 0$
 puisque $v_{b} (ac+1)>0$.
 Par le lemme de Hensel, un tel point
 se rel\`eve en un point  de
$z^p=(x^p+c)(a(x^p+c)+1) \neq 0$ dans $k_{v }$. 
Ainsi  $U(k_{v}) \neq \emptyset$.

Soit $v=v_{a}$. 
Sur le corps $\F_{v} $, on trouve une solution de $z^p=x^p+c\neq 0,$ 
donc une solution de $z^p=(x^p+c)(a(x^p+c)+1) \neq 0$. Une telle solution se rel\`eve en une solution de
$z^p=(x^p+c)(a(x^p+c)+1) \neq 0$ dans $k_{v }$. On a donc $U(k_{v }) \neq \emptyset$.
 
\medskip

D'apr\`es la proposition \ref{modelecalculbrauer}, l'alg\`ebre cyclique $A=(K/k,Q(x))=(d,Q(x))_{\zeta}$ appartient \`a $\Br(X)$.

Soit $v$ une place de $k$ et soit $L/k_{v}$ une extension finie de corps, et $w$ la valuation sur $L$.
L'application $ev_{A} : X(L) \to \Br L\subset \Q/\Z$  obtenue par \'evaluation de $A$ est continue. Par le th\'eor\`eme des fonctions implicites,
son image est la m\^eme que l'image de l'\'evaluation de   $A$ sur $U(L)$ :
$$ ev_{A} :U(L) \to \Br L \subset \Q/\Z, \hskip2mm P \mapsto A(P)$$

Si $d$ est une puissance $p$-i\`eme dans $L$, l'image  de $ ev_{A}$ est clairement nulle.
C'est le cas pour $v$ place complexe et aussi pour  $v$ une place r\'eelle, puisque l'on a 
 $a>>0$ et $b>>0$, et donc $d>>0$. C'est aussi le cas si $v(p)>0$.

Soit $v$ une place non archim\'edienne de $k$  distincte de $v_{a}$ et de $v_{b}$, avec $v(p)=0$.
Si $d$ est une puissance $p$-i\`eme dans $L$, l'image de $ev_{A}$ est nulle. 
Supposons que
$d$ ne soit pas une puissance $p$-i\`eme dans $L$.  L'extension de corps locaux $KL/L$ est
non ramifi\'ee, les normes sont exactement les \'el\'ements de valuation divisible par $p$. 
Soit $(\Xi,x) \in U(L)$. De l'\'equation de $U$ on d\'eduit $w(Q(x))+w(R(x)) \equiv 0 \bmod p$.
Si $w(x)<0$, alors $w(Q(x))=w(x^p+c)=pw(x)  \equiv 0 \bmod p$.  Supposons $w(x)\geq 0$.
On a $(ax^p+ac+1) - a(x^p+c)=1$. On a donc soit $0= w(ax^p+ac+1)$ soit $0= w(a)+w(x^p+c)=w(x^p+c)$.
Donc $w(Q(x))=0$ ou $w(R(x))=0$. 
Comme on a  $w(Q(x))+w(R(x)) \equiv 0 \bmod p$,
on en d\'eduit  $w(Q(x))  \equiv 0 \bmod p$. Dans tous les cas, on a donc $w(Q(x))  \equiv 0 \bmod p$,
et ceci implique $(d,Q(x))_{\zeta}=0$. On voit donc
que pour toute place $v$ non archim\'edienne de $k$  distincte de $v_{a}$ et de $v_{b}$,
et toute extension finie $L/k_{v}$,
l'application $ ev_{A} : U(L) \to \Br L \subset \Q/\Z$ a son image nulle.

Soit $v=v_{a}$. Si $w(x)<0$ alors $Q(x)=x^p+c$ est une puissance $p$-i\`eme dans $L$,
et donc $(d,Q(x))_{\zeta}=0$. Supposons $w(x)\geq 0$.  Alors $ax^p+ac+1 \equiv 1 \bmod a$,
ce qui implique que $R(x)=ax^p+ac+1$ est une puissance $p$-i\`eme dans $L$, et donc
$(d,Q(x))_{\zeta}=(d,R(x))_{\zeta}=0.$

On voit donc que pour toute place $v \neq v_{b}$ et tout z\'ero-cycle $z_{v}$ sur $X_{k_{v}}$,
on a $A(z_{v})=0$.

Soit enfin $v=v_{b}$. 
L'\'equation de $U$ donne
$$(d,Q(x))_{\zeta} + (d, R(x)_{\zeta}) = 0 \in \Z/p \subset \Q/\Z = \Br L.$$
Supposons $w(x)\neq 0$. On a $w(ac+1)>0$.
Ainsi  $R(x)=ax^p+ac+1$ est le produit de $a$ et
d'une puissance $p$-i\`eme dans $L$. Donc
$(d,R(x))_{\zeta}= (ab, a)_{\zeta} $, et donc $(d,Q(x))_{\zeta}= -(ab,a)_{\zeta} \in \Z/p$.
 Supposons $w(x)> 0$. De $w(ac+1)>0$ on d\'eduit que $c$ est une unit\'e dans $L$
 et que $-ac$ est une puissance $p$-i\`eme dans $L$.
 Alors $Q(x)=x^p+c$ est le produit de $-a^{-1}$ et d'une puissance  $p$-i\`eme dans $L$, donc
 $(d,Q(x))_{\zeta}= (ab, -a^{-1})_{\zeta}=-(ab,-a)_{\zeta}$.
 Si $p$ est impair, $-1$ est une puissance $p$-i\`eme, donc $-(ab,a)_{\zeta}=-(ab,-a)_{\zeta}$.
 Si $p=2$, on a encore cette \'egalit\'e, car $(ab,-1)_{v_{b}}=0$ \cite[Lemma 5.2]{poonen2}.
 Ceci est \'etabli en utilisant $(a,-1)_{v_{b}}=0$ (facile) et $(b,-1)_{v_{b}}=0$ (obtenu via la
 formule du produit). 
 En conclusion, pour $v=v_{b}$, toute extension finie $L$ de $k_{v_{b}}$
 et tout point $P \in U(L)$, on a $A(P)= -(ab,-a)_{\zeta}= (b^{-1},-a)_{\zeta} =  (b^{-1},a)_{\zeta} \in \Br L$.
 
 Pour toute extension finie $E/F$ de corps, 
 le compos\'e $\Br F \to  \Br E  \to \Br F$  de la restriction et de la corestriction est
 la multiplication par le degr\'e $[E:F]$.
Pour $v=v_{b}$ et $z_{v_{b}}= \sum_{i}n_{i}P_{i}$  un z\'ero-cycle sur $X_{k_{v}}$,
la valeur prise par $A$  sur ce z\'ero-cycle est  donc
$$\sum n_{i} [k(P_{i}):k]  (b^{-1},a)_{\zeta} \in \Br k_{v},$$
soit encore ${\rm deg}(z_{v_{b}})(b^{-1},a)_{\zeta}\in \Br k_{v_{b}}$.
Si le  z\'ero-cycle $z_{v_{b}} $ est de degr\'e~1, la valeur prise est
$[(b^{-1},a)_{\zeta}]_{v_{b}} \in \Br k_{v_{b}} \subset \Z/p$, et cette classe est non nulle,
car l'unit\'e $a$ n'est pas une puissance $p$-i\`eme
dans le corps r\'esiduel de $k_{v_{b}}$.

 On voit donc que pour toute famille 
 $\{z_{v}\}, v \in \Omega,$ de z\'ero-cycles de degr\'e 1 sur $X$, on a
 $$\sum_{v \in \Omega} A(z_{v})  = [(b^{-1},a)_{\zeta}]_{v_{b}} \neq 0 \in \Q/\Z.$$
\end{proof}

 \section{Le th\'eor\`eme}

  Soient  $p$ un nombre premier et $k$ un corps de nombres  contenant une racine primitive $p$-i\`eme $\zeta$ de $1$,
  qu'on fixe, d\'eterminant ainsi un isomorphisme $\Z/p \oi \mu_{p}$ sur $k$.
Soient $a,b, c, d \in k$
 et $Q(x), R(x), P(x)=Q(x)R(x)$  comme dans  la proposition \ref{exemple}
 et sa d\'emonstration.

 Soient $A=k[u,1/u]$ et $B=k[u,1/u][v]/(v^p-du)$, avec le plongement 
 \'evident $A \hookrightarrow B$, qui fait de $\Spec B$ un $A$-sch\'ema fini
\'etale, plus pr\'ecis\'ement un $\Z/p$-torseur.

Soit ${\mathbb U}$ le  $R$-sch\'ema lisse  d\'efini par
$${\rm Norm}_{B/A}(\Xi)=P(x)  \neq 0,$$
o\`u $\Xi = \sum_{i=0}^{p-1} v^{i}x_{i}$.
L'espace total ${\mathbb U}$ est affine, car donn\'e par le syst\`eme 
  $${\rm Norm}_{B/A}(\Xi)=P(x) , \hskip1mm P(x)y-1=0.$$
 La fibre de  ${\mathbb U}/ R$ en un point $u \in \Spec R$ est not\'ee ${\mathbb U}_{u}$.
 Pour $u$ un $k$-point,  ${\mathbb U}_{u}= U(k((du)^{1/p})/k, P(x))$.

En utilisant le th\'eor\`eme d'Hironaka, 
on construit un $A$-sch\'ema int\`egre  $\mathbb X$, projectif et  lisse  sur $A$, 
contenant le $A$-sch\'ema ${\mathbb U}$ comme ouvert,
 tel que
pour tout  $k$-homomorphisme $A \to k$, i.e. pour tout  choix de $u \in k^{\times}$, le $A$-sch\'ema $\mathbb X$ 
se sp\'ecialise en une  $k$-vari\'et\'e projective, lisse, g\'eom\'etriquement int\`egre  ${\mathbb X}_{u}=X_{k((du)^{1/p})/k}$ contenant
 $U(k((du)^{1/p})/k, P(x))$ comme ouvert dense.

\begin{lem}\label{courbecommune}
Soit $C$ la $k$-courbe lisse d\'efinie par l'\'equation
$$z^p=Q(x)R(x) \neq 0.$$

Il existe un ensemble fini $S$ de places de $k$, contenant les places archim\'ediennes,
tel que, pour toute place $v \notin S$,

(i) On a $C(k_{v}) \neq \emptyset$.

(ii) Pour  tout $u \in k_{v}^{\times}$   de valuation $v$-adique $v(u)$ non nulle modulo $p$,
l'alg\`ebre $(u,Q(x))_{\zeta}$, quand \'evalu\'ee sur $C(k_{v})$, prend toutes les valeurs
dans $\Z/p \subset \Br k_{v}$.
\end{lem}
\begin{proof}
Soit $v$ une place finie    telle que $Q(x)R(x)$ ait ses coefficients dans l'anneau $O_{v}$
des entiers de $k_{v}$, son coefficient dominant une unit\'e dans $O_{v}$, et
que $p$ soit inversible dans $O_{v} $. Soit $\F_{v}$ le corps r\'esiduel en $v$.
 La $\F_{v}$-courbe $C_{\gamma}$ d\'efinie  par
$$z_{1}^p=\gamma Q(x) \neq 0, \hskip1mm z_{2}^p=\gamma^{-1} R(x)\neq 0,$$
  est lisse et g\'eom\'etriquement int\`egre. Son mod\`ele projectif \'evident dans $\P^3$
  est une courbe intersection compl\`ete lisse $D_{\gamma}$ de deux surfaces de degr\'e $p$,
  le genre de $D_{\gamma}$ est donc $p^3-2p^2+1$.
  Sur une cl\^oture alg\'ebrique du corps de base, on v\'erifie imm\'ediatement que le compl\'ementaire  de $C_{\gamma}$ dans $D_{\gamma}$ est
 form\'e de $3p^2$ points.

Si le cardinal de $\F_{v}$ est plus grand qu'une constante d\'ependant seulement de  l'entier $p$, par les estimations de Weil pour les courbes,
pour tout $\gamma \in \F_{v}^{\times}$, la $\F_{v}$-courbe $C_{\gamma}$  
 contient un $\F_{v}$-point.  En associant au  point de coordonn\'ees $(x,z_{1},z_{2})$
 le point de coordonn\'ees $(x,z=z_{1}.z_{2})$ on d\'efinit 
  un morphisme de  $C_{\gamma}$ dans
 la courbe  lisse  d'\'equation $z^p=Q(x)R(x) \neq 0$ sur $\F_{v}$.
 Par le lemme de Hensel, l'image d'un $\F_{v}$-point de $C_{\gamma}$ dans
cette courbe lisse est un $\F_{v}$-point qui
se rel\`eve en un $O_{v}$-point de la courbe $C$.  Sur un tel $O_{v}$-point,
l'alg\`ebre $(u,Q(x))_{\zeta}$ a pour valeur $\gamma^{v(u)} \in \F_{v}^{\times}/\F_{v}^{\times p} \simeq \Z/p$.
Comme $\gamma \in \F_{v}^{\times}$ est  arbitraire, ceci \'etablit le lemme.
\end{proof}

L'\'enonc\'e suivant est la g\'en\'eralisation du th\'eor\`eme 1.3 de \cite{poonen} (cas $p=2$).
C'est le Th\'eor\`eme 5.2 de \cite{vav}.

\begin{prop}\label{finitude}
L'ensemble   des $u \in k^{\times} $ tels que que l'on ait \`a la fois  ${\mathbb X}_{u}(\A_{k}) \neq \emptyset$ 
et    ${\mathbb X}_{u}(\A_{k})^{\Br} =\emptyset$
 forme un nombre fini de classes dans $k^{\times}/k^{\times p}$.
\end{prop}

\begin{proof} Il est clair que chacune des deux propri\'et\'es consid\'er\'ees ne d\'epend que de la classe
de $u$ dans $k^{\times}/k^{\times p}$.  Soient $C$
 et $S$ comme dans le lemme \ref{courbecommune}.
On suppose de plus que $S$ contient les places finies  avec $v(d) \neq 0$, c'est-\`a-dire
$v_a$ et $v_b$.
Pour tout $u \in k^{\times}$, la courbe $C$
est contenue dans ${\mathbb X}_{u}$. Soit $u$ tel que 
 ${\mathbb X}_{u}(\A_{k}) \neq \emptyset$. S'il existe une place $v \notin S$ telle
 que $v(u)=v(du) \in \Z$ ne soit pas  divisible par $p$, alors, d'apr\`es le lemme \ref{courbecommune},  l'alg\`ebre $(du,Q(x))_{\zeta}$
parcourt sur $C(k_{v}) \subset {\mathbb X}_{u}(k_{v})$ toutes les valeurs dans $\Z/p \subset \Br k_{v}$. D'apr\`es la proposition \ref{modelecalculbrauer}, la classe de  $(du,Q(x))_{\zeta}$ appartient \`a $\Br {\mathbb X}_{u}$
et engendre $\Br {\mathbb X}_{u}/ \Br k$.
 Il ne saurait donc y avoir obstruction de Brauer-Manin \`a l'existence d'un  $k$-point, et encore moins d'un z\'ero-cycle de degr\'e~1
 sur ${\mathbb X}_{u}$. On voit donc que les $u \in k^{\times}$ tels que  ${\mathbb X}_{u}(\A_{k}) \neq \emptyset$
 et qu'on ait obstruction de Brauer-Manin \`a l'existence d'un z\'ero-cycle de degr\'e 1
ont leur image dans le noyau de l'application
$$ {\rm div } : k^{\times}/k^{\times p} \to \oplus_{v \notin S} \Z/p$$
et ce noyau est fini (finitude du nombre de classes et th\'eor\`eme des unit\'es de Dirichlet).
 \end{proof}

\bigskip

Nous pouvons maintenant donner  la d\'emonstration du th\'eor\`eme  principal.
\begin{theo}\label{thprincipal}
Soit $k$ un corps de nombres. Soit $r>0$ un entier. Le compl\'ement de $k^{\times r}$
dans $k^{\times}$ est un ensemble diophantien : il existe une $k$-vari\'et\'e  $Z$
et un $k$-morphisme $f: Z \to \A^1_{k}$ tel que $f(Z(k))$ soit le compl\'ementaire des puissances
$r$-i\`emes dans $k$.
\end{theo}

\begin{proof}
Une r\'eduction connue \cite[Cor. 1.2]{poonen}, \cite[Lemma 5.3 \& Proof of Cor. 1.2, p. 133]{vav} 
montre qu'il suffit d'\'etablir le r\'esultat lorsque $r$ est un nombre premier $p$, et que de plus
 $k$  contient les racines primitives $p$-i\`emes de l'unit\'e. On fixe une telle racine primitive $p$-i\`eme $\zeta$ dans $k$.

Notons $H_{v} = k^{\times} \cap k_{v}^{\times p}$ et $N_{v}$ le compl\'ementaire de $H_{v}$ dans $k^{\times}$. 
Ce sont des  ensembles diophantiens  (cf.  \cite[Thm 1.5]{poonen}) stables par multiplication par 
 $k^{\times p}$.
 
Soit $D_{1}$  l'ensemble des $u  \in k^{\times}$ tels que  ${\rm Sym}^{2p+1}{\mathbb U}_{u}(k) \neq \emptyset$.
C'est un ensemble stable par multiplication par  $k^{\times p}$. C'est un ensemble diophantien, car
c'est l'image des  points $k$-rationnels de la $k$-vari\'et\'e  ${\rm Sym}^{2p+1}_{A}{\mathbb U}$ 
par la projection naturelle ${\rm Sym}^{2p+1}_{A}{\mathbb U} \to \Spec A= \G_{m,k} \subset \A^1_{k}$.
On a not\'e ici  ${\rm Sym}^{2p+1}_{A}{\mathbb U}$ le
quotient du produit fibr\'e de $N$ exemplaires de ${\mathbb U}$ au-dessus de $\Spec A$ par l'action de
$\frak{S}_{N}$.

Soit $D_{2}$ la r\'eunion de $D_{1}$ et des $N_{v}$ pour $v \in S$, avec $S$ comme au lemme
\ref{courbecommune}.
C'est un ensemble diophantien
dans $k^{\times} \subset k$, stable par
multiplication par $k^{\times p}$.  

Consid\'erons le compl\'ement  $E$ de $D_{2}$ dans $k^{\times}$. 
C'est un ensemble stable par multiplication par  $k^{\times p}$. Pour tout $u \in E$, on a 
${\mathbb X}_{u}(\A_{k}) \neq \emptyset$.  
D'apr\`es le corollaire  \ref{caspartCTSD},
 l'ensemble $E$ est contenu dans l'ensemble $C$ des $u \in k^{\times}$ tels que de plus
${\mathbb X}_{u}(\A_{k})^{\Br} = \emptyset$.
D'apr\`es la proposition \ref{finitude}, l'ensemble $C$ est union d'un nombre {\it fini} d'orbites de   $k^{\times p}$ dans $k^{\times}$.
Il en est donc de m\^eme de $E$. D'apr\`es la proposition \ref{exemple}, l'orbite de $1$, soit $k^{\times p}$, n'est pas dans $D_{2}$, 
et elle est dans $E$.

 Le compl\'ement $D_{3}$ de
 $k^{\times p}$ dans $E$  est une union finie d'orbites  de $k^{\times p}$ dans $k^{\times}$, chacune clairement un ensemble diophantien,
 donc $D_{3}$ est un ensemble diophantien. La r\'eunion de $D_{2}$ et de $D_{3}$ est un ensemble diophantien $D_{4}$ dont le compl\'ement
 dans $k^{\times}$ est pr\'ecis\'ement $k^{\times p}$. Ainsi le compl\'ement de $k^{\times p}$ dans $k^{\times}$ 
  est l'ensemble diophantien $D_{4}$,
 ce qui \'etablit le th\'eor\`eme \ref{thprincipal}.
 \end{proof}
 
 \begin{rema}
 Dans la d\'emonstration, on peut remplacer l'ensemble diophantien  $D_{1}$ par  l'ensemble  diophantien $D'_{1}$ form\'e des 
  $u  \in k^{\times}$
  tels que  ${\rm Sym}^{2p+1}{\mathbb X}_{u}(k) \neq \emptyset$. L'ensemble $D'_{1}$ est stable par multiplication par $k^{\times p}$. Il contient $D_{1}$. On d\'efinit   $D'_{2}$ comme
   la r\'eunion de $D'_{1}$ et des $N_{v}$ pour $v \in S$, puis $E' \subset E \subset C$
   comme le compl\'ement de $D'_{2}$ dans $k^{\times}$.
 \end{rema}
 
 \section{Une variante de la  d\'emonstration}
 
 Dans la d\'emonstration du th\'eor\`eme \ref{thprincipal}, un ingr\'edient cl\'e est le corollaire \ref{caspartCTSD}
 du th\'eor\`eme \ref{caspartCTSkSD}.

 En combinant un r\'esultat non publi\'e de Salberger (1985), d\'ecrit dans l'appendice (paragraphe \ref{appendice} ci-apr\`es),
  on peut utiliser directement
  les th\'eor\`emes principaux de  \cite{CTSD} ou \cite{CTSkSD} sur les z\'ero-cycles
pour donner une variante de la d\'emonstration du th\'eor\`eme \ref{thprincipal}, variante
qui \'evite le corollaire \ref{caspartCTSD}
 du th\'eor\`eme \ref{caspartCTSkSD}, et donc aussi ce dernier th\'eor\`eme.

Le corollaire \ref{caspartCTSD} du  
th\'eor\`eme \ref{caspartCTSkSD} \'etablit le r\'esultat suivant :

\medskip

{\it Soit $u \in k^{\times}$. Si l'on a 
  ${\mathbb X}_{u}(\A_{k})^{\Br}  \neq \emptyset$, alors ${\rm Sym}^{2p+1}{\mathbb X}_{u} (k) \neq \emptyset$.}
  
  \medskip

Voici comment \'etablir directement un substitut de ce r\'esultat.
D'apr\`es \cite[Thm. 5.1]{CTSD} ou \cite[Thm. 4.1]{CTSkSD}, l'hypoth\`ese  ${\mathbb X}_{u}(\A_{k})^{\Br}  \neq \emptyset$
implique l'existence d'un z\'ero-cycle de degr\'e 1 sur  la $k$-vari\'et\'e ${\mathbb X}_{u}$. 
Le th\'eor\`eme \ref{corsalberger} ci-dessous, avec $s=2p$, donne alors 
l'existence d'un z\'ero-cycle effectif de degr\'e  $(p-1)^2$ sur ${\mathbb X}_{u}$.
Ainsi  ${\mathbb X}_{u}(\A_{k})^{\Br}  \neq \emptyset$ implique que $u \in k^{\times}$ est dans l'image des
points $k$-rationnels de ${\rm Sym}^N_{A}{\mathbb X}$ via la projection
${\rm Sym}_{A}^{p^2-2p+1}{\mathbb X} \to \G_{m}$. On proc\`ede alors exactement comme dans la d\'emonstration du
th\'eor\`eme \ref{thprincipal}, en y rempla\c cant la projection ${\rm Sym}_{A}^{p+1}{\mathbb U} \to \G_{m}$
par ${\rm Sym}_{A}^{p^2-2p+1}{\mathbb X} \to \G_{m}$.

\section{Appendice : R\'esultats d'effectivit\'e pour les z\'ero-cycles sur les familles de vari\'et\'es de Severi-Brauer}\label{appendice}

 Soit  $p$ un nombre premier, $k$ un corps contenant une racine primitive $p$-i\`eme de $1$
 et  $P(x) \in k[x]$ un polyn\^ome s\'eparable de degr\'e $2p$. Soit $K/k$ une extension cyclique de corps, de groupe $G=<\sigma> = \Z/p$.

 Soit $R_{1}$ la $k[x]$-alg\`ebre 
 $$R_{1} = \oplus_{i=0}^{p-1}  K[x] \xi^{i},$$
 avec $x$ central et  les relations $\xi^p=P(x)$ et pour $\lambda \in K$, $\xi.\lambda=\sigma(\lambda).\xi$.
Alors $R_{1}$ est    un ordre maximal, et donc en particulier h\'er\'editaire,
dans l'alg\`ebre simple centrale  $(K/k,P(x))$ sur 
sur le corps $k(x)$.

 Notons $P'(y)=y^{2p}P(1/y)$. Soit $R_{2}$ la $k[y]$-alg\`ebre 
 $$R_{1} = \oplus_{i=0}^{p-1}  K[y] \eta^{i},$$
 avec $ y$ central et  les relations $\eta^p=P'(y)$ et, pour $\lambda \in K$, $\eta.\lambda=\sigma(\lambda).\eta$.
Alors $R_{2}$ est un ordre h\'er\'editaire  sur $k[y]$.

Les identifications $y=1/x$ et $y^2\xi=\eta$ permettent, sur
$k[x,x^{-1}]$ et $k[y,y^{-1}]$, de recoller les deux ordres h\'er\'editaires   en un faisceau $\tilde{R}$ de tels ordres
  sur la droite projective $\P^1_{k}$.
  
  D'apr\`es M. Artin \cite{A},
   le foncteur $F$ sur la cat\'egorie des  $\P^1_{k}$-sch\'emas 
    dont les points $F(S)$ sur un  $\P^1_{k}$-sch\'ema affine $S$
  sont les id\'eaux \`a gauche  du faisceau d'alg\`ebres  $\tilde{R}_{S}$ qui sont  localement libres de rang $p$  
  comme faisceaux de $O_{S} $-modules,
  et qui sont  localement
  facteurs directs  comme   $O_{S} $-modules de $\tilde{R}_{S}$,
  est repr\'esentable par un $\P^1_{k}$-sch\'ema projectif $X'$.
   
   Au-dessus de l'ouvert de $\P^1_{k}$ o\`u l'alg\`ebre $\tilde{R}$ est d'Azumaya, ce sch\'ema $X'$  se
   restreint en un sch\'ema de Severi-Brauer relatif.
  Toujours d'apr\`es M. Artin \cite[Thm. (1.4)]{A}, la fibre g\'en\'erique de ce $\P^1_{k}$-sch\'ema $X'$  est
  la vari\'et\'e de Severi-Brauer sur le corps $k(\P^1)=k(x)$ associ\'ee \`a la $k(x)$-alg\`ebre simple centrale $(K/k,P(x))$.
 
 Toujours d'apr\`es M. Artin, la composante connexe du sch\'ema $X'$  contenant la fibre g\'en\'erique 
  est un $k$-sch\'ema r\'egulier $X=X(K/k,P)$, donc lisse  sur $k$ si  ${\rm car}(k)=0$.
 Les fibres de $\pi : X \to \P^1_{k}$ sont g\'eom\'etriquement connexes.

  M. Artin (\cite{A}, voir aussi \cite{F}), donne une description pr\'ecise des fibres non lisses
  du morphisme projectif $X \to \P^1_{k}$, description dont
  nous n'aurons pas besoin ici.

\medskip

Le th\'eor\`eme suivant est annonc\'e et partiellement d\'emontr\'e par P. Salberger dans sa th\`ese
 \cite{salbergerthese}.
 \begin{theo}  \cite[Cor. 5.9]{salbergerthese}\label{salbergerenonce}
Soit $k$ un corps de caract\'eristique nulle.  Soit $K=k(\P^1)$. 
Soit $A$ une alg\`ebre \`a division sur $K$ d'indice premier $p$.
Soit $\tilde R$  un faisceau d'ordres maximaux de $A$ sur $\P^1_{k}$.
Soit $\pi : X \to \P^1_{k}$ le sch\'ema comme ci-dessus.
Soit $s$ la somme des degr\'es sur $k$ des points ferm\'es de $\P^1_{k}$
 au voisinage desquels la $O_{\P^1_{k}}$-alg\`ebre  $\tilde R$ n'est pas une alg\`ebre d'Azumaya.

 Tout z\'ero-cycle sur $X$ de degr\'e au moins \'egal \`a $$N_{0}=(1/2) (p-1)(s- 2(p+1)/p)$$
est rationnellement \'equivalent sur $X$ \`a un z\'ero-cycle effectif.
\end{theo}

Ce th\'eor\`eme  g\'en\'eralise un r\'esultat sur les surfaces fibr\'ees en coniques 
au-dessus de la droite projective
 \cite{CTCo}. Salberger donne un \'enonc\'e plus g\'en\'eral, au-dessus d'une courbe
 de genre quelconque.
Sa d\'emonstration utilise
une version de l'in\'egalit\'e de Riemann pour les id\'eaux des ordres maximaux
sur une courbe due \`a Witt \cite{Witt}, et reprise en particulier
dans \cite{vBvG}.

\begin{cor}\label{corsalberger}
Avec les notations et hypoth\`eses ci-dessus, 
soit
   $N=N(p,s)\geq 1$ le plus petit entier congru \`a 1 modulo $p$ et au moins \'egal \`a $ (1/2) (p-1)(s- 2(p+1)/p)$.
 Les conditions suivantes sont \'equivalentes :

(a) $X$ poss\`ede un z\'ero-cycle de degr\'e 1.

(b) $X$ poss\`ede un z\'ero-cycle effectif de degr\'e $N$.

(c) Le produit sym\'etrique ${\rm Sym}^NX$ poss\`ede un $k$-point.
\end{cor}

Donnons un \'enonc\'e analogue pour toute  $k$vari\'et\'e projective et lisse
$k$-birationnelle \`a $X$.

Le  lemme suivant est une extension aux z\'ero-cycles du lemme de Lang--Nishimura \cite[Lemme 3.1.1]{CTCoSa}.

\begin{lem}\label{Nishicycle}
Soit $k$ un corps de caract\'eristique z\'ero.
Soit $X$ une $k$-vari\'et\'e lisse int\`egre et $Y$ une $k$-vari\'et\'e propre.
S'il existe une application $k$-rationnelle $f$  de $X$ vers $Y$, et si $X$
poss\`ede un z\'ero-cycle effectif de degr\'e  $N$, alors $Y$ poss\`ede un z\'ero-cycle effectif de degr\'e $N$.
\end{lem}

\begin{proof}
Soit $z$ un tel z\'ero-cycle. On \'eclate le support du z\'ero-cycle.
Ceci donne une $k$-vari\'et\'e lisse $X'$ et une application $k$-rationnelle $f'$ de $X'$ vers $Y$,
qui est d\'efinie sur le compl\'ementaire d'un ferm\'e de codimension~2 de $X'$.
La fibre  au-dessus d'un point ferm\'e $P$ dans le support de $z$ est un espace projectif $\P^{d-1}_{k(P)}$.
Il existe donc un z\'ero-cycle effectif $z'$ de degr\'e $N$ dans l'ouvert de d\'efinition de $f'$.
L'image de ce z\'ero-cycle par l'application $f'_{*}$ est un z\'ero-cycle de degr\'e $N$ dans $Y$.  
\end{proof}

 \begin{lem}\label{bireffectifN}
  Soit $k$ un corps de caract\'eristique z\'ero. Soient $X$ et $Y$ deux $k$-vari\'et\'es propres et lisses, g\'eom\'etri\-quement connexes,
  $k$-birationnelles entre elles.
Si pour un entier $ N>0$ tout z\'ero-cycle sur $Y$ de degr\'e   $N$ est rationnellement \'equivalent sur 
$Y$ \`a un z\'ero-cycle effectif, et si  $X$ poss\`ede un z\'ero-cycle  de degr\'e $N$, alors $X$
poss\`ede un z\'ero-cycle effectif de degr\'e $N$.
\end{lem}

\begin{proof}
 Il existe des ouverts non vides  $U \subset Y$ et $V \subset X$ qui sont $k$-isomorphes.
Soit $z$ un z\'ero-cycle sur $X$ de degr\'e   $N$. Par un lemme de d\'eplacement facile,
ce z\'ero-cycle est rationnellement \'equivalent sur $X$ \`a un z\'ero-cycle de degr\'e $N$  \`a support dans $V$.
On consid\`ere le z\'ero-cycle image dans $U$. Il est de degr\'e $N$. Sur $Y$, il est rationnellement
\'equivalent \`a un z\'ero-cycle effectif de degr\'e $N$. Le lemme \ref{Nishicycle}  \'etablit l'existence d'un z\'ero-cycle de degr\'e $N$
sur $X$. 
\end{proof}

Nous pouvons maintenant \'enoncer le corollaire suivant du r\'esultat de Salberger.

\begin{theo}\label{corsalberger}
Soit $k$ un corps de caract\'eristique z\'ero.
Soit $X\to \P^1_{k}$ un $k$-morphisme de $k$-vari\'et\'es projectives et lisses
g\'eom\'etriquement int\`egres dont la fibre g\'en\'erique est   $k$-birationnelle \`a
 une vari\'et\'e de Severi-Brauer d'indice premier $p$, d'alg\`ebre associ\'ee $A/k(\P^1)$. 
 Soit $s$ la somme des degr\'es sur $k$ 
 des points ferm\'es de $\P^1_{k}$ o\`u l'alg\`ebre $A$ a un r\'esidu non trivial.
 Soit $N=N(p,s)\geq 1$ un entier congru \`a 1 modulo $p$ et au moins \'egal \`a $ (1/2) (p-1)(s- 2(p+1)/p)$.
 Les conditions suivantes sont \'equivalentes :

(a) $X$ poss\`ede un z\'ero-cycle de degr\'e 1.

(b) $X$ poss\`ede un z\'ero-cycle effectif de degr\'e $N$.

(c) Le produit sym\'etrique ${\rm Sym}^NX$ poss\`ede un $k$-point.
  \end{theo} 
  
  \begin{proof}
  Soit  $A/K=k(\P^1)$ l'alg\`ebre simple centrale associ\'ee \`a la fibre g\'en\'erique de $\pi$.
  Si $A$ n'est pas une alg\`ebre \`a division, alors c'est une alg\`ebre de matrices, et 
 la $k$-vari\'et\'e  $X$ est $k$-birationnelle \`a $\P^{p-1}_{k} \times_{k}\P^1_{k}$, donc \`a $\P^p_{k}$.
 Le groupe de Chow  $CH_{0}(X)$ des z\'ero-cycles  modulo l'\'equivalence rationnelle sur $X$
 est alors \'egal \`a $\Z$, engendr\'e par
 la classe d'un point $k$-rationnel, et le th\'eor\`eme est clair.
 
 Supposons que $A$ est une alg\`ebre \`a division. Soit $\pi : Y \to \P^1_{k}$ un mod\`ele
 donn\'e par le th\'eor\`eme \ref{salbergerenonce}.

La consid\'eration d'une fibre lisse sur un $k$-point de $\P^1_{k}$ montre que $X$ et $Y$
poss\`edent soit un $k$-point soit un point ferm\'e de degr\'e $p$ sur $k$.
Ceci \'etablit en particulier que (b) implique (a).
  Supposons (a). D'apr\`es ce qui pr\'ec\`ede,
il existe un z\'ero-cycle de degr\'e $N(p,s)$ sur $X$. Le th\'eor\`eme \ref{salbergerenonce} et
le lemme \ref{bireffectifN} donnent alors (b).
Les \'enonc\'es (b) et (c) sont clairement  \'equivalents.
  \end{proof}


\begin{thebibliography}{99}

\bibitem{A} M. Artin, Left ideals in maximal orders, in {\it  Brauer Groups in Ring Theory and Algebraic Geometry}, LNM {\bf 917}, Springer-Verlag (1982) p. 182--193.

 

\bibitem{CTCo} J.-L. Colliot-Th\'el\`ene et  D. Coray, L'\'equivalence rationnelles sur les points ferm\'es des surfaces rationnelles fibr\'ees en coniques, Compositio math. {\bf 39} (1979) 301--322.

\bibitem{CTCoSa} J.-L. Colliot-Th\'el\`ene, D. Coray et J.-J. Sansuc,  
Descente et principe de Hasse pour certaines vari\'et\'es rationnelles, Journal f\"ur die reine und angewandte Mathematik (Crelle)  {\bf 320} (1980), 150--191. 

\bibitem{CTSaSD} J.-L. Colliot-Th\'el\`ene,  J.-J. Sansuc et Sir Peter Swinnerton-Dyer,  In tersections of two quadrics and Ch\^atelet surfaces, I, II,
Journal f\"ur die reine und angewandte Mathematik (Crelle) {\bf 373} (1987) 37--107; {\bf 374} (1987) 72--168.

\bibitem{CTSD} J.-L. Colliot-Th\'el\`ene et Sir Peter Swinnerton-Dyer, 
Hasse principle and weak approximation  
for pencils of Severi--Brauer and similar varieties,
  Journal f\"ur die reine und angewandte Mathematik (Crelle) {\bf 453} (1994) 49--112.

\bibitem{CTSkSD}  J.-L. Colliot-Th\'el\`ene, A. N. Skorobogatov et Sir Peter Swinnerton-Dyer, 
Rational points and zero-cycles on fibred varieties : Schinzel's hypothesis and Salberger's device,
  Journal f\"ur die reine und angewandte Mathematik (Crelle) {\bf 495} (1998) 1--28.


\bibitem{F} E. Frossard, Fibres d\'eg\'en\'er\'ees  des sch\'emas de Severi-Brauer d'ordres, J. Alg. {\bf 198} (1997), 362--387.

\bibitem{GS} P. Gille et T. Szamuely, Central simple algebras and Galois cohomology,  Cambridge studies in advanced mathematics {\bf 101} (2006).

\bibitem{K} J.  K\"onigsmann, Defining $\Z$ in $\Q$, {\tt{arxiv.org/abs/1011.3424v2}}, \`a para\^{\i}tre dans
Annals of Math.

\bibitem{poonen} B. Poonen, The set of nonsquares in a number field is Diophantine, Math. Res. Lett. {\bf 16} (2009),
no. 1, 165--170. 

\bibitem{poonen2} B. Poonen, Existence of rational points on smooth projective varieties, J. Eur. Math. Soc. {\bf 11} (2009) 529--543.

 

\bibitem{salbergerthese}  P. Salberger, Class groups of orders and Chow groups of their Brauer--Severi schemes, in {\it $K$-Theory of orders and their Brauer--Severi schemes}, Th\`ese, Chalmers University of Technology,  G\"oteborg, 1985.

\bibitem{salbergerinv} P. Salberger, Zero-cycles on rational surfaces over number fields, Invent. math. {\bf 91} (1988), no. 3, 505--524.


\bibitem{vBvG} M. van den Bergh et J. van Geel, Algebraic elements in division algebras over function fields of curves,  Israel J. Math. {\bf 52} (1985), 33--45.

\bibitem{vav} A. V\'arilly-Alvarado et  B. Viray, Higher dimensional analogues of Ch\^atelet surfaces, Bull. London Math. Soc.
  {\bf 44} (2012), no. 1, 125--135. 


\bibitem{Witt} E. Witt, Riemann-Rochscher Satz und Z-Funktion im Hyperkomplexen, Math. Ann. {\bf 110}  (1934)  12--28.

\bibitem{w} O. Wittenberg, Z\'ero-cycles sur les fibrations au-dessus d'une courbe de genre quelconque,
Duke Math. J. {\bf 161} (2012) 2113--2166.

\end{thebibliography}
\end{document}